\documentclass{amsart}
\usepackage[utf8]{inputenc}
\usepackage{amsmath}
\usepackage{amsfonts}
\usepackage{amsthm}
\usepackage{amssymb}
\usepackage{mathtools}
\usepackage{enumitem}
\usepackage{xcolor}
\usepackage{hyperref}
\newcommand{\F}{\mathcal{F}}

\newcommand{\Q}{\mathcal{Q}}

\renewcommand{\H}{\mathcal{H}}
\renewcommand{\P}{\mathcal{P}}
\newcommand{\V}{\mathcal{V}}
\newcommand{\foc}{\mathfrak{foc}}
\newcommand{\Inn}{\operatorname{Inn}}
\newcommand{\Aut}{\operatorname{Aut}}
\newcommand{\Out}{\operatorname{Out}}
\newcommand{\Hom}{\operatorname{Hom}}

\newcommand{\id}{\operatorname{id}}

\renewcommand{\phi}{\varphi}

\newcommand{\gen}[1]{\langle #1 \rangle}

\theoremstyle{theorem}
\newtheorem{theorem}{Theorem}[section]

\newtheorem{lemma}[theorem]{Lemma}

\theoremstyle{definition}
\newtheorem{definition}[theorem]{Definition}

\theoremstyle{remark}
\newtheorem{remark}[theorem]{Remark}

\begin{document}
\title{Realizing finite groups as automizers}
\author{Sylvia Bayard} 
\email{sbayard@ucsc.edu}
\address{Mathematics Department\\UC Santa Cruz\\1156 High Street\\Santa Cruz, CA 95064}
\author{Justin Lynd}
\email{lynd@louisiana.edu}
\address{Department of Mathematics\\University of Louisiana at Lafayette\\Lafayette, LA 70504}
\thanks{J.L. was partially supported by NSF Grant DMS-1902152.}
\date{\today}
\begin{abstract}
It is shown that any finite group $A$ is realizable as the automizer in a
finite perfect group $G$ of an abelian subgroup whose conjugates generate $G$.
The construction uses techniques from fusion systems on arbitrary finite
groups, most notably certain realization results for fusion systems of the type
studied originally by Park.
\end{abstract}

\maketitle

\section{Introduction}

Not every finite group is realizable as $\Aut(U)$ for some finite group $U$.
For example, no nontrivial cyclic group of odd order is the automorphism group
of a group.  We study here the realization of finite groups by automizers of
subgroups of finite groups. That is, given a finite group $A$, we study when it
is possible to find a finite group $G$ and a subgroup $U \leq G$ such that $A
\cong \Aut_G(U) = N_G(U)/C_G(U)$.  As it stands, the answer to this question is
``always possible'' for trivial reasons: choose a faithful action of $A$ on an
elementary abelian $p$-group $U$ (for some prime $p$), and take for $G$ the
semidirect product of $U$ by $A$.  In this case, $U$ is normal in $G$. Our main
result shows that it is possible to realize $A$ as $\Aut_G(U)$ where $U$ is
very far from being normal.

\begin{theorem}
\label{T:main}
For each finite group $A$, there exist a finite perfect group $G$ and a
homocyclic abelian subgroup $U$ of $G$ such that $\left< U^G\right> = G$ and
$\Aut_G(U) \cong A$. 
\end{theorem}
Here, we write $\gen{U^G}$ for the normal closure of $U$ in $G$, the subgroup
of $G$ generated by the $G$-conjugates of $U$. A group $G$ is perfect if it
coincides with its commutator subgroup. A homocyclic abelian group is a direct
product of isomorphic cyclic groups. 

We do not know whether more restrictions can be placed on $G$, up to and
including whether $G$ can be taken to be simple. Likewise, we do not know if
whether more restrictions can be placed on $U$, such as requiring $U$ to be an
elementary abelian $p$-group for some prime $p$.

Ultimately, the group $G$ is constructed fairly explicitly as the commutator
subgroup of a wreath product of the form $(U \rtimes A) \wr \Sigma_n$, but the
embedding of $U$ in $G$ is not an obvious one.  The method for constructing $G$
and the embedding of $U$ relies on certain basic constructions in fusion
systems on arbitrary finite groups. A fusion system on a finite group $S$ (not
necessarily a $p$-group) is a category with objects the subgroups of $S$, and
with morphism sets consisting of injective homomorphisms between subgroups,
subject to two weak axioms which we recall in Section~\ref{S:background}. The
standard example is the fusion system $\F_S(G)$ of the group $G$ on the finite
subgroup $S$ in which the morphisms are the $G$-conjugation homomorphisms
between subgroups of $S$. The most important ingredient in the construction
here is a result of Sejong Park realizing fusion on finite $p$-groups
\cite{Park2010, Park2016}, but which we require in the more general setting of
fusion systems on finite groups, where it is due to Warraich
\cite{Warraich2019}. 

\begin{theorem}[{\cite[Section~4]{Warraich2019}, c.f. \cite[Theorem~1.1]{Park2016}}]
\label{T:park}
For each fusion system $\F$ on a finite group $S$, there is a finite group $G$
containing $S$ as a subgroup and such that $\F = \F_S(G)$. 
\end{theorem}

\"Unl\"u and Yal\c{c}in also considered fusion systems on finite groups with an
eye on Park's result \cite[Section~5]{UnluYalcin2012}, although they did not
prove Theorem~\ref{T:park}. For the convenience of the reader, we provide a
sketch of proof of Theorem~\ref{T:park} which is modeled closely on Park's
proof for $S$ a $p$-group. For example, the $G$ of Theorem~\ref{T:park} is the
group of automorphisms as a right $S$-set of a certain $S$-$S$ biset associated
with $\F$, similarly as in \cite{Park2010} and \cite{Park2016}.

In order to use Theorem~\ref{T:park} to prove Theorem~\ref{T:main}, we need to
be able to construct a suitable finite group $S$ and fusion system $\F$ on $S$.
One consequence of the way this fusion system is built is the following result. 

\begin{theorem}
\label{T:realfus}
For each finite group $A$, there are a finite group $S$, a homocyclic abelian
subgroup $U$ of $S$, and a fusion system $\F$ on $S$ such that $\foc(\F) = S$,
$Q(\F) = 1$, and $\Aut_\F(U) \cong A$. 
\end{theorem}

The definition of the focal subgroup $\foc(\F)$ of an arbitrary fusion system
is given in Section~\ref{S:background} and is identical to the definition for
fusion systems on $p$-groups. The definition of the subgroup $Q(\F)$ of $S$,
which is a sort of replacement for $O_p(\F)$ in a fusion system over an
arbitrary finite group when compared with a saturated fusion system over a
$p$-group, is also given there.

When $A$ is a $p$-group for some prime $p$, $S$ is also a $p$-group in the
construction we present.  But we do not know whether it is possible to choose
$\F$ to be a fusion system on a $p$-group independently of $A$ in
Theorem~\ref{T:realfus}, much less whether $\F$ can be taken to be a
\emph{saturated} fusion system on a $p$-group. 

A MathOverflow question of Peter Mueller asks \cite{NormSym}: is every finite
group of the form $N_{\Sigma_n}(U)/U$ for a subgroup $U$ of some finite
symmetric group $\Sigma_n$? This work arose out of an attempt to say something
about that question. 

Here is a brief outline of the paper and some remarks on notation. In
Section~\ref{S:background} we give some background on fusion systems and
semicharacteristic bisets and give a definition of $Q(\F)$. We also write down
a proof of the existence of semicharacteristic bisets for fusion systems on
arbitrary finite groups and provide a discussion of Theorem~\ref{T:park}. In
Section~\ref{S:realaut}, we prove a slightly more detailed version of
Theorem~\ref{T:realfus} and combine it with Theorem~\ref{T:park} to prove
Theorem~\ref{T:main}. We use left-handed notation for conjugation $x \mapsto {
}^gx = gxg^{-1}$. Our iterated commutators are right-associated: $[X,Y,Z] =
[X,[Y,Z]]$, etc. We sometimes write $G'$ for the commutator subgroup of a group
$G$. 

\section{Fusion systems on finite groups, semicharacteristic bisets, and the
Park embedding}\label{S:background}

\subsection{Fusion systems}\label{SS:fusion}
 
\begin{definition} 
Let $S$ be a finite group. A \emph{fusion system} on $S$ is a category
$\mathcal{F}$ with objects the set of subgroups of $S$, subject to the
following two axioms: for all $P, Q \le S$,
\begin{enumerate} 
\item $\textup{Hom}_\mathcal{F}(P, Q)$ consists of a set of injective
homomorphisms from $P$ to $Q$, including all such morphisms induced by
$S$-conjugation.
\item Each $\varphi \in \textup{Hom}_\mathcal{F}(P, Q)$ is the composite of an
$\mathcal{F}$-isomorphism from $P$ to $\varphi(P)$ and the inclusion from
$\varphi(P)$ to $Q$.  
\end{enumerate} 
\end{definition}
Axiom (1) implies that any inclusion $\iota_P^Q$ of subgroups $P \leq Q$ is a
morphism in $\F$ from $P$ to $Q$ (being conjugation by $1 \in S$). Therefore, a
morphism can be restricted to any subgroup of the source. Axiom (2) then
implies for example that the target of any morphism can be restricted to a
subgroup containing the image. 

If $G$ is a group and $S$ is a finite subgroup of $G$, there is a fusion system
$\F_S(G)$ of $G$ on $S$ with morphism sets $\Hom_G(P,Q) = \{c_g\colon t \mapsto
{ }^gt \mid { }^gP \leq Q\}$ consisting of the $G$-conjugation homomorphisms
mapping $P$ into $Q$.  This is the standard example of a fusion system.
Theorem~\ref{T:park} shows that indeed every fusion system on $S$ is of this
form, and $G$ can be taken finite. 

The notation $\Aut_\F(P)$ is short for $\Hom_\F(P,P)$ in a fusion system $\F$
on $S$.  When $\F = \F_S(G)$ for some group $G$ and $P \leq S$, then
$\Aut_\F(P) = \Aut_G(P)$ from the definitions. 

We introduce now several properties of subgroups and morphisms in a fusion
system that we will need, many of which are identical to their counterparts for
fusion systems on $p$-groups \cite{AschbacherKessarOliver2011, CravenTheory}.

\begin{definition}[Generation of fusion systems]
Let $S$ be a finite group and let $\mathfrak{X}$ be a set of injective
homomorphisms between subgroups of $S$. The \emph{fusion system on $S$
generated by $\mathfrak{X}$}, denoted $\gen{\mathfrak{X}}_S$, is the
intersection of the fusion systems on $S$ containing $\mathfrak{X}$.
\end{definition}

If $\F_1$ and $\F_2$ are two fusion systems on the finite group $S$, then the
category $\F_1 \cap \F_2$ with objects the subgroups of $S$ and with morphism
sets $\Hom_{\F_1 \cap \F_2}(P,Q) := \Hom_{\F_1}(P,Q) \cap \Hom_{\F_2}(P,Q)$ is
again a fusion system on $S$.  Thus, the definition makes sense.  As in the
case of fusion systems on finite $p$-groups, it is easy to see that an
injective group homomorphism is in $\gen{\mathfrak{X}}_S$ if and only if it can
be written as a composition of restrictions of homomorphisms in $\Inn(S) \cup
\mathfrak{X}$. 

\begin{definition}[Direct products]
\label{D:directproduct}
Let $S_1$ and $S_2$ be finite groups, and let $\F_i$ be a fusion system on
$S_i$, $i = 1, 2$. The direct product $\F_1 \times \F_2$ is the fusion system
over $S_1 \times S_2$ generated by the homomorphisms $(\phi_1, \phi_2) \colon
P_1 \times P_2 \to S_1 \times S_2$, where $\phi_i \in \Hom_{\F_i}(P_i, S_i)$.
\end{definition}

We also need the definition of the focal subgroup of a fusion system.

\begin{definition}[The focal subgroup]
Let $\F$ be a fusion system on the finite group $S$. The \textit{focal subgroup
of} $\F$ is the subgroup of $S$ generated by elements of the form $[\phi,s] :=
\varphi(s) s^{-1}$, where $s \in S$ and $\varphi\colon \gen{s} \to S$ is a
morphism in $\F$.
\end{definition}

\begin{remark}\label{R:focal}
By the Focal Subgroup Theorem \cite[7.3.4]{Gorenstein1980}, if $G$ is a finite
group and $S$ is a Sylow $p$-subgroup of $G$, then $\foc(\F_S(G)) = S \cap
[G,G]$. When $S$ is an arbitrary subgroup of $G$, there is the obvious
inclusion $\foc(\F_S(G)) \leq S \cap [G,G]$ since each generating element
$\phi(s)s^{-1} \in S$ is a commutator $gsg^{-1}s^{-1} = [g,s]$ for some $g \in
G$, but in general the reverse inclusion need not hold. 
\end{remark}

\subsection{Nonextendable morphisms and the subgroup $Q(\F)$}

\begin{definition}[Nonextendable morphisms]
Let $\F$ be a fusion system on the finite group $S$, and let $P, Q \leq S$.  A
morphism $\phi \in \Hom_\F(P,Q)$ is said to be \emph{nonextendable} if it does
not extend to a morphism defined on any subgroup of $S$ properly containing
$P$. That is, whenever $P \leq R \leq S$ and $\tilde{\phi} \in \Hom_\F(R,S)$ is
such that $\iota_Q^S \circ \phi = \tilde{\phi} \circ \iota_{P}^{R}$, then $R =
P$. 
\end{definition}

\begin{definition}\label{D:QF}
For a fusion system $\F$ on a finite group $S$, define $\Q(\F)$ to be the set
of all subgroups $Q$ of $S$ for which there is a nonextendable morphism $\phi
\colon Q \to S$ in $\F$, and let $Q(\F)$ be the intersection of the family
$\Q(\F)$.
\end{definition}

The relevance of the subgroup $Q(\F)$ will be seen later in
Lemma~\ref{L:capQi}.  By the same proof as for fusion systems on $p$-groups, if
$\F$ is a fusion system on a finite group $S$, there is a unique largest
subgroup $N$ of $S$ having the property that each morphism $\phi \colon P \to
Q$ in $\F$ extends to a morphism $\tilde{\phi}\colon PN \to QN$ with
$\tilde{\phi}|_N(N) = N$, which we might denote by $O_S(\F)$. (If $S$ is a
$p$-group, then this is the largest normal $p$-subgroup $O_p(\F)$ of $\F$.) It
follows from the definitions that $O_S(\F)$ is a subgroup of each member of
$\Q(\F)$, and so $O_S(\F) \leq Q(\F)$.  Thus, the requirement of $Q(\F) = 1$ in
Theorem~\ref{T:realfus} is stronger than a requirement of $O_S(\F) = 1$. 

\begin{remark}\label{R:QF}
The direct product $\Aut_\F(S) \times \Aut_\F(S)$ acts on the set of pairs
$(Q,\phi)$ consisting of a subgroup $Q \in \Q(\F)$ and a nonextendable morphism
$\phi \colon Q \to S$ via $(\alpha,\beta)\cdot (Q,\phi) = (\alpha(Q), \beta\phi
\alpha^{-1})$. In particular, $Q(\F)$ is $\Aut_\F(S)$-invariant.
\end{remark}

\subsection{Semicharacteristic bisets}\label{SS:bisets}
For a finite group $S$, an $S$-$S$-biset $X$ is a set with left and right
$S$-action such that $(ux)v = u(xv)$ for all $u, v \in S$, $x \in X$. An
$S$-$S$-biset can also be viewed as an $(S\times S)$-set via $(u,v)\cdot x =
uxv^{-1}$. For a subgroup $Q \le S$ and a homomorphism $\varphi\colon Q
\rightarrow S$, let
\begin{equation*}
S \times_{(Q, \varphi)} S = (S \times S)/\sim
\end{equation*}
where $(xu, y) \sim (x, \varphi(u)y)$ for $x, y \in S$, $u \in Q$, and let
$\left <x, y\right >$ be the equivalence class of $(x, y)$. The group action is
defined for $t \in S$ by $t\left<x,y\right> = \left<tx,y\right>$ and
$\left<x,y\right>t$ = $\left<x,yt\right>$. Importantly, $S \times_{(Q,
\varphi)} S$ is also isomorphic to $(S \times S)/\Delta(Q, \varphi)$ as
$S\times S$-sets, where $\Delta(Q, \varphi) := \{(u, \varphi(u)): u \in Q\}$.
We refer to $\Delta(Q,\varphi)$ as a \emph{twisted diagonal subgroup}.

\begin{definition}[{cf. \cite[Definition~1.2]{Park2016}}]
Let $\mathcal{F}$ be a fusion system on a finite group $S$. A \textit{left
semicharacteristic biset} for $\mathcal{F}$ is a finite $S$-$S$-biset $X$
satisfying the following properties.
\begin{itemize}
\item $X$ is $\mathcal{F}$\textit{-generated}, i.e., every transitive subbiset
of $X$ is of the form $S \times_{(Q, \varphi)} S$ for some $Q \le S$ and some
$\varphi \in \text{Hom}_\mathcal{F} (Q, S)$
\item $X$ is \textit{left} $\mathcal{F}$\textit{-stable}, i.e., $_Q X \cong
{}_\varphi X$ as $Q$-$S$-bisets for every $Q \le S$ and every $\varphi \in
\Hom_{\F}(Q, S)$, where ${}_\varphi X$ is the $Q$-$S$ biset whose
left action is induced by $\varphi$.
\end{itemize}
\end{definition}

In \cite{Park2016}, Park showed that each fusion system over a $p$-group has a
semicharacteristic biset. Then Warraich \cite{Warraich2019} extended this to
fusion systems on arbitrary finite groups. We give a proof here for the
convenience of the reader. 

\begin{theorem}\label{T:leftsemichar}
Let $\mathcal{F}$ be a fusion system on the finite group $S$.  Then there
exists a left semi-characteristic biset $X$ for $\mathcal{F}$, and $X$ can be
chosen to include at least one $S$-$S$ orbit of the form $S \times_{(S,\alpha)}
S$ for each $[\alpha] \in \Out_\F(S)$. 
\end{theorem}
\begin{proof}
The construction of $X$ is the same as that of \cite{BrotoLeviOliver2003},
similarly as in \cite{Park2016} and does not depend on $S$ being a $p$-group.
We give the details for the convenience of the reader.  

Let $\F' = \F \times \F_S(S)$ be the product fusion system. Observe that the
set of subgroups of the form $\Delta(P,\phi)$ with $P \leq S$ and $\phi \in
\Hom_\F(P,S)$ is closed under $\F'$-conjugacy and taking subgroups. For
example, $(\psi,c_s) \in \Hom_{\F'}(\Delta(P,S), S\times S)$ sends
$\Delta(P,S)$ to $\Delta(\psi(P),c_s\phi\psi^{-1})$.

Let
\begin{equation*}
    X_0 = \sum_{[\alpha] \in \Out_\F(S)} S \times_{(S, \alpha)} S,
\end{equation*}
where the sum denotes disjoint union.  This is a $\F$-generated virtual $S$-$S$
biset with nonnegative rational coefficients having the property that
$|X_0^{\Delta(S,\beta)}| = |N_{S \times S}(\Delta(S,\id))/\Delta(S,\id)| =
|Z(S)|$ for all $\beta \in \Aut_\F(S)$.  Thus, the fixed point sizes in $X_0$
are constant on $\F'$-conjugacy classes of twisted diagonal subgroups
$\Delta(S,\beta)$ with $\beta \in \Aut_\F(S)$. 

Let $\H$ be a set of subgroups of $S \times S$ of the form $\Delta(P,\phi)$
with $\phi\colon P \to S$ in $\F$ such that $\H$ is closed under
$\F'$-conjugacy and taking subgroups. Assume given inductively an
$\F$-generated virtual $S$-$S$ biset $X_0$ with nonnegative rational
coefficients such that fixed point sizes on $X_0$ are constant on
$\F'$-conjugacy classes of twisted diagonal subgroups which are not in $\H$.
Let $\P$ be an $\F'$-conjugacy class in $\H$ whose members are maximal under
inclusion among the subgroups in $\H$, and let $\Delta(P,\phi) \in \P$ be a
subgroup for which the fixed point set $X_0^{\Delta(P,\phi)}$ has largest size
(among the elements of $\P$). Define
\[
X_1 = X_0 + \sum_{\Delta} \frac{|X_{0}^{\Delta(P,\phi)}| -
|X_0^{\Delta}|}{|N_{S \times S}(\Delta)/\Delta|} (S \times S)/\Delta
\]
where the sum runs over a set of representatives $\Delta$ for the subgroups in
$\P$ up to $S \times S$-conjugacy.  Thus, $X_1$ is an $\F$-generated virtual
$S$-$S$ biset with nonnegative rational coefficients.  A subgroup $D \notin
\H-\P$ has a fixed point on $(S \times S)/\Delta$ if and only if $D$ is $S
\times S$-conjugate to some $\Delta$, and in this case the number of such fixed
points is $|N_{S \times S}(\Delta)/\Delta|$. So by construction, $|X_1^D| =
|X_0^D|$ for each subgroup $D \leq S \times S$ which is not in $\H$, and
$|X_1^D| = |X_0^{\Delta(P,\phi)}|$ for each $D \in \P$. In particular, fixed
point sizes on $X_1$ are constant on $\F'$-conjugacy classes of twisted
diagonal subgroups which are not in $\H-\P$. 

By induction there is thus an $\F$-generated virtual $S$-$S$ biset
$X_{\mathbb{Q}}$ with nonnegative rational coefficients such that fixed points
sizes are constant on $\F'$-conjugacy classes of twisted diagonal subgroups.
Let $m$ be a positive integer such that $X := mX_{\mathbb{Q}}$ has integer
coefficients. Then $X$ is an $\F$-generated $S$-$S$ biset with the same
property. 

It remains to show that $X$ is left $\F$-stable. Let $\phi \colon Q \to S$ be a
morphism in $\F$, and let $D \leq Q \times S$ be a subgroup. Then either $D$
and $(\phi,\id)(D)$ are not twisted diagonal subgroups in which case they have
no fixed points on $X$, or $|({ }_Q X)^D| = |X^D| = |X^{(\phi,\id)(D)}| = |({
}_\phi X)^{D}|$ by construction of $X$. This shows the fixed point sizes on ${
}_QX$ and on ${ }_\phi X$ are equal for any subgroup of $Q \times S$.  So ${
}_Q X \cong { }_\phi X$ as $Q \times S$-sets. 
\end{proof}

The inductive nature of the argument in the proof of
Theorem~\ref{T:leftsemichar} sometimes makes it difficult to understand
precisely which orbits of the form $S \times_{(Q,\phi)} S$ actually occur in a
semicharacteristic biset. The following lemma gives a sufficient condition on a
pair $(Q,\phi)$ which forces the inclusion of the corresponding orbit. 

\begin{lemma}\label{L:nonextendable}
Let $X$ be a left semicharacteristic biset for $\F$ containing an orbit of the
form $S \times_{(S,\id)} S$. If $\varphi \in \Hom_\mathcal{F}(P,S)$ is
nonextendable, then $X$ contains an orbit isomorphic to $S \times_{(P,\phi)}
S$.
\end{lemma}
\begin{proof}
Since $\Delta(P,\id) \leq \Delta(S,\id)$ and $\Delta(S,\id)$ fixes a point in
$S \times_{\Delta(S,\id)}S$, $\Delta(P,\id)$ fixes a point in $X$. So
$\Delta(P,\phi)$ fixes a point, say $x \in X$, because $X$ is left $\F$-stable.
Let its orbit be isomorphic to $(S \times S)/\Delta(Q, \gamma)$, and identify
$x$ with the coset $(x_1, x_2)\Delta(Q, \gamma)$. The stabilizer of $x$ is then
$\Delta(Q^{x_1}, c_{x_2}\gamma c_{x_1}^{-1})$, so $\Delta(P,\phi) \leq
\Delta(Q^{x_1}, c_{x_2}\gamma c_{x_1}^{-1})$. This means that $\phi$ extends to
the morphism $c_{x_2}\gamma c_{x_1}^{-1}$ in $\F$ defined on $Q^{x_1}$. Since
$\phi$ is nonextendable, $P = Q^{x_1}$ and $\phi = c_{x_2}\gamma
c_{x_1}^{-1}$. Thus, $X$ contains the $S$-$S$ orbit $S \times_{(Q,\gamma)} S
\cong S\times_{(P,\phi)}S$.
\end{proof}

\begin{lemma}\label{L:QF}
Let $X = \sum_{i = 1}^k S \times_{(Q_i,\phi_i)} S$ be a left semicharacteristic
biset for a fusion system $\F$ on a finite group $S$. Then 
\[
\bigcap_{i=1}^k \bigcap_{s \in S} { }^sQ_i \leq Q(\F).
\]
\end{lemma}
\begin{proof}
Let $\Q(X) = \{ { }^sQ_i \mid 1 \leq i \leq k, s \in S\}$ and $Q(X) = \bigcap
\Q(X)$ for short.  Thus, we must show $Q(X) \leq Q(\F)$.  By
Lemma~\ref{L:nonextendable}, for each nonextendable morphism $\phi \colon Q \to
S$ in $\F$, there is some point of $X$ with stabilizer $\Delta(Q,\phi)$ in $S
\times S$. So for each $s \in S$, there is some point in $X$ with stabilizer
$\Delta({ }^sQ, c_s \phi c_{s}^{-1})$ and $c_s \phi c_{s}^{-1}$ is
nonextendable by Remark~\ref{R:QF}.  This shows $\Q(\F) \subseteq \Q(X)$, so
$Q(X) \leq Q(\F)$.
\end{proof}

The reverse inclusion in Lemma~\ref{L:QF} need not hold. If $X$ is a left
semicharacteristic biset for $\F$, then the disjoint union of $X$ with a number
of free $S \times S$-orbits $(S \times S)/\Delta(1,\iota_1^S)$ is again left
semicharacteristic. So there is always a semicharacteristic biset with some
$Q_i = 1$. 

\subsection{The Park embedding} 

Let $\F$ be a fusion system on the finite group $S$, and let $X$ be a left
semicharacteristic biset for $\F$ which contains an orbit of the form $S
\times_{(S,\id)} S$. Consider the group $G = \Aut( { }_1 X)$ of automorphisms
of $X$ as a right $S$-set.  We explain briefly Park's embedding of $S$ into $G$
with respect to which conjugation in $G$ on the subgroups of $S$ realizes the
fusion system $\F$. 

Fix a decomposition
\begin{equation*}
    X = \sum_{i=1}^k S \times_{(Q_i, \varphi_i)} S
\end{equation*}
such that $Q_i \le S$ and $\varphi_i \in \Hom_\F(Q_i, S)$ for all $1 \le i \le
k$ and such that $Q_1 = S$ and $\varphi_1 = \id_S$.  Following \cite{Park2010},
define $\iota$ as: 
\begin{equation*}
    S \xrightarrow{\iota} \textrm{Aut}(_1X) = G
\end{equation*}
\begin{equation*}
    u \mapsto (x \mapsto ux).
\end{equation*}
This is indeed an injection because each orbit $S \times_{(Q_i,\phi)} S$ is
free as a left $S$-set.  The same argument from \cite[Theorem~3]{Park2010}
copied verbatim shows that $\iota$ induces an isomorphism of fusion systems $\F
\cong \F_{\iota(S)}(G)$. This gives Theorem~\ref{T:park}.

\begin{theorem}[{\cite[Chapter~4]{Warraich2019}, c.f. \cite{Park2010}}]
\label{T:parkdetail}
Let $\F$ be a fusion system on the finite group $S$, and let $X$ be any left
semicharacteristic biset for $\F$ which contains the orbit $S \times_{(S, \id)}
S$. Let $G = \Aut({}_1X)$, the group of automorphisms of $X$ as a right
$S$-set. Then $G \cong S \wr \Sigma_n$ for some natural number $n$, and there
is an injection $\iota: S \rightarrow G$ such that $\mathcal{F} \cong
\mathcal{F}_{\iota(S)}(G)$.
\end{theorem}

We next set up notation that will be needed later, looking more closely at the
structure of $G$ and the embedding $\iota$.  For each $i$, fix a collection
$\{t_{ij}\}_{j \in J_i}$ of representatives of the left cosets of $Q_i$, and
set $n_i = |S : Q_i| = |J_i|$. The action of $u \in S$ on the coset
representatives is given by $ut_{ij}Q_i = t_{i\sigma_i(u)(j)}Q_i$, where
$\sigma_i(u)\colon J_i \to J_i$ is the induced permutation on $J_i$.  As a
right $S$-set, the biset $S \times_{(Q_i, \varphi_i)} S$ decomposes as
\begin{equation*} S \times_{(Q_i, \varphi_i)} S = \sum_{j \in J_i}\left<t_{ij},
S\right>, \end{equation*}
where $\left<t_{ij}, S\right> := \{\left<t_{ij}, y\right> \; | \; y \in S\}$ is
the set of ordered pairs with free and transitive right $S$-action given by
$\gen{t_{ij},y}\cdot s$ = $\gen{t_{ij},ys}$. Hence, also
\begin{equation*}
    X = \sum_{i=1}^k \sum_{j \in J_i}\left<t_{ij}, S\right>
\end{equation*}
as a right $S$-set.  

Since the right action of $S$ on $\gen{t_{ij},S}$ is regular, each automorphism
of $\gen{t_{ij},S}$ as a right $S$-set is left multiplication by an element of
$S$, i.e., of the form $\gen{t_{ij},y} \mapsto \gen{t_{ij}, sy}$. Thus, $\Aut({
}_1\gen{t_{ij}, S}) \cong S$. It therefore follows from the above
decompositions that
\begin{equation*}
    G_i := \text{Aut}(_1(S \times_{(Q_i, \varphi_i)} S)) \cong S \wr \Sigma_{n_i},
\end{equation*}
and 
\begin{equation*}
    G = \Aut({ }_1 X) \cong S \wr \Sigma_{n}.
\end{equation*}
where $n = \sum_{i =1}^k n_i$.

We examine more closely the map $\iota$. Now $S$ acts from the left on each $S
\times_{(Q_i, \varphi_i)} S$, so $\iota(S) \le \prod G_i \le G$. Let $u \in S$.
Since $ut_{ij} \in t_{i\sigma_i(u)(j)}Q_i$, we have
$(t_{i\sigma_i(u)(j)})^{-1}ut_{ij} \in Q_i$, and
\begin{equation*}
u\left<t_{ij}, y\right> 
= \left<ut_{ij}, y\right> 
= \left<t_{i\sigma_i(u)(j)}\cdot (t_{i\sigma_i(u)(j)})^{-1}ut_{ij},\,\, y\right> 
= \left<t_{i\sigma_i(u)(j)}, \,\,\varphi_i((t_{i\sigma_i(u)(j)})^{-1}ut_{ij})y\right>.
\end{equation*}
Thus, writing $\pi_i$ for the projection $\Pi G_i \to G_i$, we have
\begin{equation*}
\pi_i(\iota(u)) = (\, (\,\,
\varphi_i((t_{i\sigma_i(u)(j)})^{-1}ut_{ij})\,\,)_{j \in J_i};
\,\,\sigma_i(u)\,) \in S \wr \Sigma_{n_i}. 
\end{equation*} 
The following lemma gives some information on the intersection of $\iota(S)$
with the base subgroup of $G$. 

\begin{lemma}\label{L:capQi}
Let $\F$ be a fusion system on the finite group $S$ with left
semicharacteristic biset $X$ containing $S \times_{(S,\id)} S$ and Park
embedding $\iota \colon S \to G = \Aut({ }_1X) \cong S\wr \Sigma_n$.  Let $B =
S^n$ be the base subgroup of $G$. Then 
\begin{equation*}
B \cap \iota(S) \leq \iota(Q(\F)). 
\end{equation*}
\end{lemma}
\begin{proof}[Proof]
Write $X = \sum_{i = 1}^k S \times_{(Q_i,\phi_i)} S$.  For each $u \in S$, the
image $\iota(u) \in B$ if and only if  $\sigma_i(u) = 1$ for all $1 \leq i \leq
k$ in the notation above. That is, $\iota(u) \in B$ if and only if $u$ fixes
all cosets $tQ_i$, that is, if and only if $u \in \bigcap_{i} \bigcap_{t \in S}
Q_i^t$.  The result now follows from Lemma~\ref{L:QF}. 
\end{proof}

\section{Proof of Theorems~\ref{T:realfus} and \ref{T:main}}\label{S:realaut}
We now state and prove a slightly more detailed version of
Theorem~\ref{T:realfus}.

\begin{theorem}
\label{T:realfus-detailed}
Let $A$ be a finite group. Then there are a finite group $S$, a fusion system
on $S$, and a homocyclic abelian subgroup $U$ of $S$ such that $Q(\F) = 1$,
$\foc(\F) = S$, and $\Aut_\F(U) = A$. Moreover, $S$, $U$, and $\F$ can be
chosen so as to satisfy the following additional properties.
\begin{enumerate}
\item[\textup{(i)}] $S$ is the semidirect product of $U$ by $A$ with respect to
a faithful action of $A$ on $U$,
\item[\textup{(ii)}] the exponent of $U$ is the exponent of $A$, and
\item[\textup{(iii)}] if $A > 1$, then there is $Q \in \Q(\F)$ such that $|S:Q| > 2|A|$. 
\end{enumerate}
\end{theorem}
\begin{proof}
In case $A = 1$, we take $G = S = U = 1$ and $\F = \F_S(G)$.  
So we may assume $A \neq 1$.  Let $e$ be the exponent of $A$.  Consider the
homocyclic group $U = C_e^{|A|} \times C_e^{|A|}$ with free action of $A$ on
each $C_e^{|A|}$ factor, and let $S := UA$ be the semidirect product with
respect to this action. Thus, $\Aut_{S}(U) \cong A$ and (i) and (ii) are
satisfied.  Let $\mathcal{V}$ be the collection of all rank 2 homocyclic
subgroups of $S$ of order $e^2$, and define
\[
\F = \gen{\Aut(V) \mid V \in \V}_S.
\]

We first prove that $\Aut_\F(U) = \Aut_{S}(U) \cong A$. By definition of a
fusion system, $\Aut_S(U) \subseteq \Aut_\F(U)$. Let $\psi \in
\text{Aut}_\mathcal{F}(U)$. By construction of $\F$ we may choose a natural
number $n$ and automorphisms $\psi_1,\dots,\psi_n$ of subgroups $T_i \leq S$
such that $\psi = \psi_n|_{U_n} \circ \ldots \circ \psi_{1}|_{U_1}$ for certain
subgroups $U_i$ which are isomorphic to $U$, and such that either $T_i = S$ or
$T_i \in \V$ for each $i = 1,\dots, n$. Since $A$ is not the trivial group, $U
\cong U_i$ has rank at least $4$, and so is not isomorphic to a subgroup of any
$V \in \V$. Thus, we must have $T_i = S$ for each $i$, and hence $\psi \in
\Aut_S(U)$.

Consider the collection of triples $(V, R, \alpha)$ such that $V < R \leq S$,
$V \in \V$, and $\alpha \in \Hom_\F(R, S)$, and such that there is an element
$c \in V$ with $\alpha(c)$ not $S$-conjugate to $c$. We claim this is the empty
collection. Assume false, and among all such triples, choose one such that
$\alpha$ has a decomposition with a minimal number $n$ of morphisms, and then
choose such a minimal decomposition $\alpha = \alpha_n|_{R_{n-1}} \circ \cdots
\circ \alpha_1|_{R_{0}}$ with $R_0 = R$ and $\alpha_i \in \Aut(T_i)$ ($T_i \in
\{S\} \cup \V$).  By definition of $\V$ and assumption on the structure of $V$,
$V$ is maximal under inclusion in $\V$. So $T_0 = S$ and $\alpha_1 \in
\Inn(S)$. Replace the triple $(V,R,\alpha)$ by $(\alpha_1(V), R_1,
\alpha\alpha_1|_{R}^{-1})$.  As $\alpha_1 \in \Inn(S)$, we have that
$\alpha_1(c)$ is not $S$-conjugate to $\alpha\alpha_1^{-1}(\alpha_1(c)) =
\alpha(c)$.  Moreover, $\alpha\alpha_1|_V^{-1}$ extends to $\alpha_n|_{R_{n-1}}
\circ \cdots \circ \alpha_2|_{R_{1}}$ on $R_1 = \alpha_1(R) > \alpha_1(V)$.
Thus, $(\alpha_1(V), R_1, \alpha\alpha_1|_V^{-1})$ is another counterexample in
which the morphism has a shorter decomposition.  This contradicts the choice of
the triple $(V, R, \alpha)$. In particular, this shows that if $V \in \V$ and a
morphism $\alpha \in \Hom_\F(V,S)$ has the property that $c$ and $\alpha(c)$
are not $S$-conjugate, then $\alpha$ is nonextendable.

Let 
\[
\V_1 = \{V \in \V \mid V \text{ supports an nonextendable automorphism}\}.
\]
We next claim that $\gen{V \mid V \in \V_1} = S$ and $\bigcap_{V \in \V_1} V =
1$.  Since the focal subgroup of $\F$ contains $[V,\Aut(V)] = V$ for each $V
\in \V$, this will also show $\foc(\F) = S$ and thus complete the proof.  

By the structure of $U$ as an $A$-module, $C_U(A) = Z_1Z_2$ with $\gen{z_1} =
Z_1 \cong C_e \cong Z_2 = \gen{z_2}$ and $Z_1 \cap Z_2 = 1$.  Since $|A| \geq
2$, $U$ has rank at least $4$. So there is a choice of a pair of cyclic
subgroups $W_1, W_2 \leq U$ of order $e$ such that $W_i \cap C_U(A) = 1$ and
$W_1Z_1 \cap W_2Z_2 = 1$. For any such choice, there is an automorphism
$\alpha_i$ of $W_iZ_i$ which interchanges $W_i$ and $Z_i$ and thus does not
extend to an $S$-automorphism of $W_iZ_i$ (because $Z_i \leq Z(S)$). So by the
above, we see that $W_iZ_i \in \V_1$ for $i = 1, 2$. In particular this shows
that $U \leq \gen{\V_1}$ and $\bigcap \V_1 = 1$. The method of proof also shows
(iii) is satisfied, since $Z_1Z_2 \in \Q(\F)$ and $|S:Z_1Z_2| \geq |U:Z_1Z_2|
\geq e^2|A| > 2|A|$. 

Let $p$ be a prime dividing $|A|$, let $p^a$ be the $p$-part of the exponent of
$A$, and let $C$ be any cyclic subgroup of $A$ with generator $c$ of order
$p^b$. We claim that there is $V \in \V_1$ with $UC/U \leq UV/U$.  Let $u \in
U-[C,U]$ be any element of order $p^{a-b}$. Then $uc$ has order $p^{a}$.  Fix
an element $w \in C_U(C)$ of order $e/p^a$ and set $W = \gen{wuc}$. Since
$e/p^a$ is prime to $p$, $W$ is cyclic of order $e$. Since the rank of $C_U(A)$
is $2$, we can again find a cyclic subgroup $Z \leq C_U(A)$ of order $e$ with
$W \cap Z = 1$, and then $V = WZ$ is homocyclic of order $e^2$. As before,
there is an automorphism of $V$ interchanging $W$ and $Z$, which therefore does
not extend to an $S$-automorphism of $V$. This shows that $V \in \V_1$. By
construction $UC/U \leq UV/U$, and we saw above that $U \leq \gen{V_1}$.
Since $C$ was an arbitrary cyclic subgroup of $p$-power order, and the set of
such subgroups generates $A$ as $p$ ranges over the primes dividing $A$, it
follows that $\gen{\V_1} = S$.
\end{proof}

Before giving the proof of Theorem~\ref{T:main}, we prove a specialized lemma
about the commutator subgroup of a wreath product. 
\begin{lemma}
\label{L:wrcomm}
Let $S$ be a group, let $K$ be a subgroup of $\Sigma_n$ with $n > 1$, and let
$\Gamma = S \wr K$ with base subgroup $B$ and $G = \Gamma' = [\Gamma,\Gamma]$.
Assume that $K'$ is perfect and transitive. Then $[B,B] \leq [K,B] = [K',B] =
[K',K',B]$ and $G = [K',B]K'$ is perfect.  
\end{lemma}
\begin{proof}
Although $n > 1$ was assumed initially, the further assumptions give implicitly
that $n \geq 5$.  Write $e_i(s)$ for the element $(1,\dots,1,s,1,\dots,1) \in B
= S_1 \times \cdots \times S_n$ (with $s$ in the $i$-th place), and $c^i_j(s) =
e_{j}(s)e_i(s)^{-1}$.  Let $J$ be any transitive subgroup of $\Sigma_n$.  Then
$c^{i}_j(s) = [g,e_{i}(s)] \in [J,B]$ for each element $g \in J$ which sends
$i$ to $j$, so $c^i_{j}(S) \in [J,B]$ for each $i$ and $j$. For $i$ an index
taken modulo $n$ and for $s$, $t \in S$, 
\[
[c^{i-1}_i(s), c^{i}_{i+1}(t^{-1})] = e_{i}([s,t]).
\]
This shows that $[S_i,S_i] \leq [J,B]$ for each $i$, and hence $[B,B] = [S,S]^n
\leq [J,B]$. Under the same assumptions on $J$, we just saw $[J,B]$ contains
all $c^i_j(s)$ with $s \in S$ and $1 \leq i, j \leq n$. These generate
$\ker(\pi)$, where $\pi \colon B \to S/S'$ is the homomorphism sending an
element of $B$ to the product of its components.  Since each generating element
of $[J,B]$ is clearly in this kernel, we have $[J,B] = \ker(\pi)$. In
particular, $[B,B] \leq [K',B] = [K,B]$. 

Next, for any subgroup $J$, we have $[J,B,J] = [J,J,B]$.  So if $J' = J$, then
$[J,B] = [B, J] = [B, J'] = [B,J,J] \leq [J,J,B] \leq [J,B]$, the first
inclusion by the Three Subgroups Lemma \cite[Theorem~2.3(ii)]{Gorenstein1980}.
So $[J,J,B] = [J,B]$. In particular, $[K',K',B] = [K',B]$ since $K'$ was
assumed perfect.

Applying \cite[Theorem~2.1]{Gorenstein1980} for example to $\Gamma = KB$, we
see that $G = \Gamma' = K'[K,B][B,B]$, and then $G = K'[K',B]$ as $[B,B] \leq
[K,B] = [K',B]$. Keeping in mind that $[[K',B], [K',B]] \leq [B,B]$ since
$B/[B,B]$ is abelian, a similar argument gives $G' = [K',K'][K',K',B] = K'[K'B]
= G$, so $G$ is perfect. 
\end{proof}

\begin{proof}[Proof of Theorem~\ref{T:main}]
Let $A$ be any finite group. If $A = 1$, then we take $G = U = 1$, so we may
assume $A > 1$. Fix a fusion system $\F$ on a finite group $S = U \rtimes A$
with $U$ homocyclic and $A$ faithful on $U$, satisfying the conclusion of
Theorem~\ref{T:realfus-detailed}.  Let $X = \sum_{i = 1}^k S
\times_{(Q_i,\phi_i)} S$ be any left semicharacteristic biset for $\F$ as in
Theorem~\ref{T:leftsemichar}, and write $\iota\colon S \to \Gamma = \Aut({ }_1X)
\cong S \wr \Sigma_n$ for the Park embedding, so that $\F \cong
\F_{\iota(S)}(\Gamma)$ via $\iota$. Set $G = \Gamma'$. To ease notation, we
identify $S$ with its image in $\Gamma$, and so we identify $\F$ and
$\F_S(\Gamma)$.  By choice of $\F$, we know $S \cap B \leq Q(\F) = 1$ by
Lemma~\ref{L:capQi} where $B$ is the base subgroup of $\Gamma$ as usual, while
also $S = \foc(\F) \leq S \cap G$ by Remark~\ref{R:focal}.  Thus, $U \leq S
\leq G$. Since $\F = \F_S(\Gamma)$, we have $\Aut_\Gamma(U) = \Aut_\F(U) \cong
A$ again by choice of $\F$, and $\Aut_S(U) \cong A$ by construction.  Since $S
\leq G$, this shows $\Aut_{G}(U) \cong A$.  We want to verify that $G$
satisfies the conclusion of the theorem.  

Let $H$ be the alternating subgroup of $\Sigma_n$, the usual complement of $B$
in $\Gamma$, and let $N = \gen{U^G}$ be the normal closure of $U$ in $G$. We
will see below that $n \geq 5$, so $H$ is simple. By Lemma~\ref{L:wrcomm}, $G$
is perfect and $G = H[H,B]$. Thus, it remains to show that $N = G$. 

Recall from the discussion of the Park embedding that $n = \sum_{i = 1}^k
|S:Q_i|$, so by Theorem~\ref{T:realfus-detailed}(iii) and
Lemma~\ref{L:nonextendable}, there is $i$ such that $n \geq |S:Q_i| > 2|A|$. So
indeed, $n \geq 5$ and $H$ is simple.  Use Bertrand's postulate to get a prime
$p$ with $|A| < p < 2|A|$, and so a prime $p$ that divides $|H|$ but not $|A|$.
By Theorem~\ref{T:realfus-detailed}(ii), $p$ divides $|H|$ but not $|S|$, so
$|H|^2$ does not divide $|G|$. On the other hand, $U \cap B \leq S \cap B \leq
Q(\F) = 1$, so as $H$ is simple, $N$ projects modulo $B$ onto $H$. Thus $|H|$
divides $|N|$.  Since $N \cap H$ is normal in $H$, we have $H \leq N$ or $H
\cap N = 1$. In the latter case, $G$ contains the subgroup $HN$ of order
divisible by $|H|^2$, a contradiction, and hence $H \leq N$.  As $H \leq N$,
$N$ contains the normal closure of $H$ in $G$, which is $H[H,B] = G$, and this
completes the proof of the theorem. 
\end{proof}

\bibliographystyle{amsalpha}{ }
\bibliography{mybib}
\end{document}